\documentclass[oneside,english]{amsart}
\usepackage[T1]{fontenc}
\usepackage[latin9]{inputenc}
\usepackage{geometry}
\usepackage{amstext}
\usepackage{amsthm}
\usepackage{amssymb}

\theoremstyle{definition}
\newtheorem{theorem}{Theorem}[section]
\newtheorem{lemma}[theorem]{Lemma}
\newtheorem{cor}[theorem]{Corollary}
\newtheorem{rem}[theorem]{Remark}

\newcommand{\Z}{\mathbb{Z}}
\newcommand{\R}{\mathbb{R}}
\newcommand{\C}{\mathbb{C}}
\newcommand{\PP}{\mathcal{P}}

\newcommand{\innerp}[2]{\big \langle #1 , \, #2\big \rangle}
\newcommand{\innerpsmall}[2]{ \langle #1 , \, #2 \rangle}
\newcommand*\diff{\mathop{}\!\mathrm{d}}

\DeclareMathOperator{\vol}{vol}

\usepackage{hyperref}
\hypersetup{colorlinks,%
citecolor=red,%
linkcolor=blue,%
}

\begin{document}


\title[The null set of a polytope, and the Pompeiu property for polytopes]
{The null set of a polytope,\\ and the Pompeiu property for polytopes}

\begin{abstract}
We study the null set $N(\mathcal{P})$ of the Fourier-Laplace transform of a polytope $\mathcal{P} \subset \mathbb{R}^d$, and we find that  $N(\mathcal{P})$ does not contain (almost all) circles in $\mathbb{R}^d$. As a consequence, the null set does not contain the algebraic varieties $\{z \in \mathbb{C}^d \mid z_1^2 + \dots + z_d^2 = \alpha^2\}$ for each fixed $\alpha \in \mathbb{C}$, and hence we get an explicit proof that the Pompeiu property is true for all polytopes. Our proof uses the Brion-Barvinok theorem, which gives a concrete formulation for the Fourier-Laplace transform of a polytope, and it also uses properties of Bessel functions. 

The original proof that polytopes (as well as other bodies) possess the Pompeiu property was given by Brown, Schreiber, and Taylor~\cite{brown73} for dimension $2$. Williams~\cite[p. 184]{williams76} later observed that the same proof also works for $d>2$ and, using eigenvalues of the Laplacian, gave another proof valid for $d \geq 2$ that polytopes have the Pompeiu property.
\end{abstract}


\thanks{F.C.M was supported by grant \#2017/25237-4, from the S\~ao Paulo Research Foundation (FAPESP). This work was partially supported by Conselho Nacional de Desenvolvimento Cient\'\i fico e Tecnol\'{o}gico --- CNPq (Proc. 423833/2018-9) and by the Coordena\c{c}\~ao de Aperfei\c{c}oamento de Pessoal de N\'\i vel Superior --- Brasil (CAPES) --- Finance Code~001.}

\date{April 5, 2021}

\author{Fabr\'\i cio Caluza Machado}
\address{F.C. Machado, Instituto de Matem\'atica e Estat\'\i stica, Universidade de S\~ao Paulo\\ Rua do Mat\~ao 1010, 05508-090 S\~ao Paulo/SP, Brazil.}
\email{fabcm1@gmail.com}

\author{Sinai Robins}
\address{S. Robins, Instituto de Matem\'atica e Estat\'\i stica, Universidade de S\~ao Paulo\\ Rua do Mat\~ao 1010, 05508-090 S\~ao Paulo/SP, Brazil.}
\email{sinai.robins@gmail.com}  

\maketitle

\section{Introduction}

The Pompeiu problem is a fundamental problem that initially arose by intertwining the basic theory of convex bodies with harmonic analysis.  To describe it precisely, consider the group $M(d)$ of all rigid motions of $\R^d$,  including translations, and fix any convex body $\PP \subset \R^d$ with $\dim \PP = d$. In 1929, Pompeiu~\cite{pompeiu29a, pompeiu29b} asked the following question. Suppose that all of the following integrals vanish:
\begin{equation}\label{Pompeiu question}
    \int_{\sigma(\PP)} f(x) \diff x = 0,
\end{equation}
taken over all rigid motions $\sigma \in M(d)$. Does it necessarily follow that $f= 0$?  

If the answer is affirmative, then the convex body $\PP \subset \R^d$ is said to have the Pompeiu property. It is a conjecture that in every dimension, balls are the {\em only} convex bodies that do not have the Pompeiu property. As is immediately apparent, the Pompeiu property is equivalent to the claim that the integral of $f$ over $\PP$, as well as the integrals of $f$ over all the rigid motions of $\PP$, {\it uniquely determine} the function $f$. 
 
It is rather surprising that after almost $100$ years,  the Pompeiu problem remains unsolved for general convex bodies in $\R^d$.  There are, however, infinite families of convex bodies which are known to have the Pompeiu property, and we recall some of these results. 

More attention has been devoted to dimension $d=2$, and a breakthrough occurred with the results of Brown, Schreiber, and Taylor~\cite{brown73}, who showed that the Pompeiu problem is very closely related to mean periodic functions, developed by L. Schwartz \cite{Schwartz}. In  \cite[Theorem 5.11]{brown73} the authors prove that any Lipschitz curve in the plane with a `corner' has the Pompeiu property, and consequently all polygons have the Pompeiu property.   Williams~\cite{williams76} mentions that the proof of Theorem~5.11 in~\cite{brown73} generalizes directly to $d$-dimensions, though such a proof is not explicitly given there. Moreover, Williams~\cite{williams76} also proves that if a convex body does not have the Pompeiu property, then its boundary must be real-analytic, and as a consequence large infinite families of convex bodies have the Pompeiu property, including polytopes. 

Despite these advances, even in dimension $2$ the Pompeiu problem remains open for general convex bodies. On the other hand, there has been a lot of interesting work that relates the Pompeiu problem to other branches of Mathematics, such as the recent work of Kiss, Malikiosis, Somlai, and Vizer~\cite{kiss}, where a discretized version of the Pompeiu problem is shown to be closely tied to the (unsolved) Fuglede conjecture over finite abelian groups.

It turns out that the Pompeiu problem is equivalent to a few other long-standing problems. One of these equivalences is the celebrated conjecture of Schiffer in pde's, relating the Pompeiu problem directly to eigenvalues of the Laplacian (see e.g., Section 3 of Berenstein~\cite{berenstein80b}).

When we consider the Fourier-Laplace transform of the body $\PP$, a very useful necessary and sufficient condition arises. To describe it precisely, suppose we are given the indicator function $1_{\PP}$ of a polytope $\PP$. We define the Fourier-Laplace transform of $\PP$ by 
\[
\hat 1_{\mathcal{P}}(z) := \int_{\mathcal{P}} e^{-2\pi i \langle x, z \rangle} \diff x, 
\]
for all $z \in \C^d$, with the inner product $\langle x, z \rangle:= x_1 z_1 + \cdots x_d z_d$ (we note that this is not the Hermitian inner product). The null set of the Fourier-Laplace transform of a polytope $\PP$ is defined by
\begin{equation*}
N(\PP) := \{\xi \in \C^d \mid \hat 1_{\PP}(\xi) = 0\},
\end{equation*}
which we also refer to simply as the null set of $\PP$. We define the complex algebraic variety
\begin{equation*}
S_{\C}(\alpha):= \{z\in \C^d \mid  z_1^2 + \cdots + z_d^2 = \alpha^2\},
\end{equation*}
for each fixed $\alpha \in \C$.  

\begin{theorem}[Brown, Shreiber, and Taylor \cite{brown73}]\label{thm:BST}
A convex body $\PP\subset \R^d$ has the Pompeiu property if and only if the Fourier-Laplace transform of $\PP$, namely $\hat 1_\PP(z)$, does not vanish identically on any of the complex varieties $ S_{\C}(\alpha)$, for any $\alpha \in \C \setminus \{0\}$.
\end{theorem}

In other words, Pompeiu's problem is equivalent to the claim that the null set $N(\PP)$ does not contain any of the complex algebraic varieties $ S_{\C}(\alpha)$. The authors of \cite{brown73} prove this condition for dimension $d=2$, and they mention that the same proof works in general dimension. Bagchi and Sitaram \cite[p. 74--75]{bagchi90} reprove Theorem \ref{thm:BST}, for $d=2$, and they also mention that the same proof works for general dimension. Berenstein comments (Section 3 of~\cite{berenstein80b}), that the condition `$\alpha \in \C \setminus \{0\}$' from Theorem~\ref{thm:BST} can be replaced by `$\alpha > 0$' (possibly under the condition that $\PP$ is simply-connected).  We do not use this restriction in our proof, however, since our arguments work for any complex `$\alpha \in \C \setminus \{0\}$'.

One direction of Theorem~~\ref{thm:BST} is easy to see. If $S_\C(\alpha) \subset N(\PP)$ for some $\alpha \in \C \setminus \{0\}$, then taking $\xi \in S_\C(\alpha)$ and letting $f(x) := e^{-2\pi i \innerpsmall{x}{\xi}}$, we have $\int_{\sigma(\PP)}f(x) \diff x = 0$ for all $\sigma \in M(d)$. For the other direction, first we notice that it is apparent that $S_\C(0) \not\subset N(\PP)$, because the zero element $0 \in S_\C(0)$, yet $0 \notin N(\PP)$ since $\hat 1_\PP(0) = \vol(\PP) \not=0$. Berenstein~\cite[p. 133]{berenstein80b} observes that in~\cite{brown73}, Brown, Schreiber, and Taylor show that if $\PP$ doesn't have the Pompeiu property, then $\hat 1_{\sigma(\PP)}$ has a common zero $z$ for all $\sigma \in M(d)$.  Next, using the fact that for a rotation $\sigma \in \mathrm{SO}(d,\R) \subset M(d)$ we get $\hat{1}_{\sigma(\PP)}(z) = \hat{1}_{\PP}(\sigma^{-1}z)$, we obtain that the orbit $\mathrm{SO}(d, \R)z \subset N(\PP)$. Letting $\alpha := z_1^2 + \dots + z_d^2$, we have that $\mathrm{SO}(d, \R)z$ is a real submanifold of $S_\C(\alpha)$, on which the analytic function $\hat{1}_{\PP}$ vanishes, hence it  also vanishes on the rest of $S_\C(\alpha)$ (see e.g., Lemma 3.1.2 in \cite{lebl19}).

\bigskip
Here we prove, in an explicit manner, that the Pompeiu property is true for all polytopes $\mathcal P \subset \R^d$, with $d \geq 2$.  In other words, we give a new proof that all polytopes have the Pompeiu property, which is simple and is essentially self-contained. In addition, the present methods allow us to prove slightly more:  `most' circles in $\R^d$ are not contained in the null set $N(\PP)$ (stated precisely in Theorem \ref{thm:whenPisapolytope}).

By way of comparison, the machinery developed in \cite{williams76}, from which it also follows that polytopes have the Pompeiu property,  is highly non-trivial;   the present proof uses an explicitly known form of the Fourier-Laplace transform of a polytope, and is much simpler. Our main result is as follows.

\begin{theorem}\label{thm:whenPisapolytope}
Let $\PP \subset \R^d$ be a $d$-dimensional polytope, $H \subset \R^d$ be a $2$-dimensional real subspace that is not orthogonal to any edge from $\PP$, and fix an orthonormal basis $\{e, f \} \subset \R^d$ for $H$. 

Then the null set $N(\PP)$ does not contain the `circle'
\begin{equation*}
C_\alpha := \big\{ \alpha (\cos t) e + \alpha (\sin t) f \in \C^d \mid t \in [-\pi, \pi] \big\},
\end{equation*}
for any $\alpha \in \C \setminus \{0\}$. 
\end{theorem}

\medskip
\noindent
As an immediate consequence of Theorem \ref{thm:whenPisapolytope} and Theorem \ref{thm:BST}, we recover William's result \cite{williams76} for polytopes, as follows.
\begin{cor}
The null set $N(\PP)$ does not contain the complex variety $S_{\C}(\alpha)$, for any $\alpha \in \C \setminus \{0\}$. Consequently, all polytopes in $\R^d$ have the Pompeiu property, for each $d\geq 2$. 
\end{cor}

\bigskip
\begin{rem}
We note that as a special case of Theorem~\ref{thm:whenPisapolytope}, the null set $N(\PP)$ does not contain any real circle sitting in $\R^d$, except perhaps for those circles that lie in some two-dimensional hyperplane orthogonal to some edge of $\PP$. The reason we cannot yet exclude this (zero measure) set of circles is because of the singularities that come from the denominators in Brion's formula for the Fourier transform of a polytope. However, because these singularities are removable, we conjecture that no circle in $\R^d$ is contained in the null set $N(\PP)$. 
\end{rem}

\bigskip
\section{Preliminaries}

\subsection{Fourier-Laplace transform of a polytope via Brion's theorem}\label{sec:brion}

In this section we recall some standard definitions from the literature, especially of tangent cones of polytopes, and their Fourier-Laplace transforms. 

Given a $d$-dimensional polytope $\PP \subset \R^d$ with vertex set $V(\PP)$, for each  $v \in V(\PP)$ we denote by $K_v$ its tangent cone, defined by
\[K_v := \{v + \lambda (x-v) \mid  x \in P,\, \lambda \geq 0\}.\]
This is a pointed cone with apex $v$ and it has a set of generators $w^v_{1}, \dots, w^v_{m}$, so that it can also be written as $K_v = \{v + \lambda_1 w^v_{1} + \dots + \lambda_m w^v_{m} \mid \lambda_j \geq 0\}$.  Each $w^v_{k}$ is a $1$-dimensional edge of $P$, emanating from $v$.  When $m = d$, we say that the cone is simplicial and we define 
\[
\det K_v := |\det(w^v_1, \dots, w^v_d)|.
\]
Every pointed cone can be triangulated into simplicial cones with no new generators, which means a collection $K_{v,1}, \dots, K_{v,M_v}$ of simplicial cones with disjoint interiors such that $K_v = \bigcup_{j=1}^{M_v} K_{v,j}$ (see Beck and Robins~\cite[Section 3.2]{beck15}). 

The Fourier-Laplace transform of a polytope $\PP$ is the entire function $\hat 1_{\PP} \colon \C^d \to \C$ defined by
\[
\hat 1_{\PP}(z) := \int_{\PP} e^{-2\pi i \innerpsmall{\xi}{z}} \diff \xi,
\] 
where $\innerp{\xi}{z} := \xi_1 z_1 + \dots + \xi_d z_d$. The same integral can also be considered over a cone $K$ instead of a polytope, but then that integral over the unbounded domain $K$ would converge only on a restricted complex domain (see Barvinok~\cite[Chapter 8]{barvinok08} for a presentation of these integrals as exponential valuations on polyhedra). The Fourier-Laplace transform of the cones $K_v$ and the polytope $P$ are related by the following striking theorem, originally due to Brion~\cite{brion88} and extended to arbitrary polytopes by Barvinok~\cite{barvinok92}, which for some polytopes produces an effective method to compute $\hat 1_{\PP}(z)$.
\begin{theorem}[Brion-Barvinok]\label{thm:Brion}
Let $\PP \subset \R^d$ be a $d$-dimensional polytope. For each $v \in V(\PP)$, there exist functions $s_v(z) := e^{-2\pi i \innerpsmall{z}{v}} q_v(z)$, where $q_v(z)$ is a rational function homogeneous of degree $-d$, such that
\begin{equation}\label{eq:Brion}
\hat 1_{\PP}(z) = \int_{\PP} e^{-2\pi i \innerpsmall{\xi}{z}} \diff \xi = \sum_{v \in V(\PP)} s_v(z),
\end{equation}
holds for all $z \in \C^d$ that are regular for all $s_v$. If $z$ is such that $\mathrm{Im}( \innerpsmall{\xi-v}{z}) < 0$ for all $\xi \in K_v \setminus \{v\}$, then 
\begin{equation}\label{eq:Brion-cone}
s_v(z) = \int_{K_v} e^{-2\pi i \innerpsmall{\xi}{z}} \diff \xi.
\end{equation}
\end{theorem}

Equation~\eqref{eq:Brion-cone} enables us to derive an explicit formula for $s_v(z)$, especially in the case when $K_v$ is  simplicial. Continuing to denote the generators of a simplicial $K_v$ by $w^v_{1}, \dots, w^v_{d}$, we have
\[
\int_{K_v} e^{-2\pi i \innerpsmall{\xi}{z}} \diff \xi = \frac{e^{-2\pi i\langle v, z \rangle}}{(2\pi i)^d} \, \frac{\det K_v}{\innerp{w_{1}^v}{z} \dots \innerp{w_{d}^v}{z}},
\]
for $z$ such that $\mathrm{Im}( \innerp{\xi-v}{z}) < 0$ for all $\xi \in K_v \setminus \{v\}$.

The condition ``$\mathrm{Im}( \innerp{\xi-v}{z}) < 0$ for all $\xi \in K_v \setminus \{v\}$'' is used to guarantee the convergence of the integral \eqref{eq:Brion-cone} to the function above; however, in~\eqref{eq:Brion} the latter formula for $s_v(z)$ can be used for all $z \in \C^d$ for which the denominators do not vanish.

More generally, we may triangulate vertex tangent cones as follows. If $K_{v,1}, \dots, K_{v,M_v}$ are simplicial cones with disjoint interiors such that $K_v = \bigcup_{j=1}^{M_v} K_{v,j}$ and for each $1 \leq j \leq M_v$,  $w_{j,1}^v, \dots, w_{j,d}^v$ are the edges of $K_{v,j}$, then for $z$ such that $\mathrm{Im}( \innerp{\xi-v}{z}) < 0$ for all $\xi \in K_v \setminus \{v\}$,
\[
\int_{K_v} e^{-2\pi i \innerpsmall{\xi}{z}} \diff \xi = \sum_{j = 1}^{M_v} \frac{e^{-2\pi i \innerpsmall{v}{z}}}{(2\pi i)^d} \frac{\det K_{v,j}}{\innerp{w_{j,1}^v}{z} \dots \innerp{w_{j,d}^v}{z}}.
\]
Therefore, we have
\begin{equation}\label{eq:fourier-polytope-brion}
\hat{1}_{P}(z) = \sum_{v \in V(P)} \sum_{j = 1}^{M_v} \frac{e^{-2\pi i \innerpsmall{v}{z}}}{(2\pi i)^d} \frac{\det K_{v,j}}{\innerp{w_{j,1}^v}{z} \dots \innerp{w_{j,d}^v}{z}}.
\end{equation}
Since $\PP$ is compact and $\hat 1_{\PP}(z)$ is continuous for all $z \in \C^d$, the formula above can be used to evaluate $\hat 1_{\PP}(z)$ for all $z \in \C^d$;  however, care has to be taken when choosing $z$ that makes any of the denominators of \eqref{eq:fourier-polytope-brion} vanish, but an appropriate limiting procedure can take care of these cases as well. 

\bigskip
\subsection{Some properties of the Bessel functions}

The Bessel functions are a very well known family of functions that appear in physical problems with spherical or cylindrical symmetry.  One reason for their ubiquity is their appearance as solutions of the wave equation when put into spherical or cylindrical coordinate systems. 

Here we collect some of their useful properties, all of which can be found e.g. in the Chapter 9 from the book of Temme~\cite{temme96}. We will be interested in the Bessel functions of the first kind, called $J_n(z)$, which are defined for complex values of $z$, and integer order $n$ (although they may also be defined for complex~$n$). They appear in the present work since they have the following integral representation:
\begin{equation*}
J_n(z) = \frac{1}{2\pi} \int_0^{2\pi}e^{iz\sin t}e^{-int}\diff t.
\end{equation*}
This identity implies that they are the coefficients of the Fourier series expansion of $e^{iz\sin t}$:
\begin{equation}\label{eq:Bessel-generating}
e^{iz\sin t} = \sum_{n\in \Z} J_n(z) e^{in t},
\end{equation}
an identity that is also known as the Jacobi-Anger expansion. Another representation for $J_n(z)$ is the hypergeometric series
\begin{equation*}
J_n(z) = \Big(\frac{z}{2}\Big)^n \sum_{k=0}^\infty \frac{(-1)^k}{(n+k)!k!} \Big(\frac{z}{2} \Big)^{2k},
\end{equation*}
from which it easily follows that $J_n(-z) = (-1)^n J_n(z)$, and also that there is the following asymptotic behavior for large $n$ and fixed $z$:
\begin{equation}\label{eq:bessel-asymptotic}
\lim_{n \to \infty}  J_n(z) \left( \frac{1}{n!}\left(\frac{z}{2}\right)^n \right) ^{-1}= 1.
\end{equation}

\bigskip
\section{Proof of Theorem~\ref{thm:whenPisapolytope}}

We divide the proof into two parts using the following lemma.

\begin{lemma} \label{lm:zero relation 1, polytopes}
Let $\PP \subset \R^d$ be a polytope oriented in such a way that no edge vector has both of its first two coordinates zero. For each vertex $v \in V(\PP)$, represent its first two coordinates in polar form:
\[
v = (r_v\cos\phi_v, r_v\sin\phi_v, v_3, \dots, v_d). 
\]

Let $Q$ be the intersection of the plane generated by the first two coordinates of $\C^d$, with the null set $N(\PP)$. If $Q$ contains a `circle' 
\[C_\alpha' := \{(\alpha \cos t, \alpha \sin t, 0, \dots, 0) \mid t \in [-\pi, \pi]\}\] 
for some $\alpha \in \C \setminus \{0\}$, then there exist $N$ and coefficients $c_{v,k} \in \C$ for $-N \leq k \leq N$, not all of them zero, so that $\alpha$ satisfies the following identity for every $n \in \Z$:
\begin{equation}\label{identity 1, vanishing Bessels-dimd}
\sum_{v \in V(\PP)} e^{-in \phi_v} \sum_{k = -N}^{N} c_{v,k} \, J_{n-k}(2\pi \alpha r_v)  i^k e^{ik \phi_v } = 0.
\end{equation}
\end{lemma}
\begin{proof}
As mentioned in Section~\ref{sec:brion}, Brion's theorem gives us Equation~\eqref{eq:fourier-polytope-brion}, valid for any $z \in \C^d$ that doesn't annul some denominator:
\begin{equation}\label{eq:hatpolytope}
\hat{1}_{\PP}(z) = \sum_{v \in V(\PP)} \sum_{j = 1}^{M_v} \frac{e^{-2\pi i \innerpsmall{v}{z}}}{(2\pi i)^d} \frac{\det K_{v,j}}{\innerp{w_{j,1}^v}{z} \dots \innerp{w_{j,d}^v}{z}}.
\end{equation}

We parameterize $C_\alpha'$ as $z(t) = (z_1, \dots, z_d) \in \C^d$, with
\begin{equation}\label{eq:2d-par}
\begin{aligned}
z_1 = \alpha  \cos t,\quad z_2 = \alpha  \sin t,\quad z_3 = \dots = z_d = 0,
\end{aligned}
\end{equation}
for $t \in (-\pi,\pi]$.

Substituting $\cos t = (e^{i t} + e^{-i t})/2$, $\sin t = (e^{i t} - e^{-i t})/(2i)$ in~\eqref{eq:2d-par} and using the assumption  that the directions $w_{j,l}^v$ do not have both of their first two coordinates equal to zero, we may see each factor $\innerp{w_{j,l}^v}{z(t)}$ as a trigonometric polynomial of degree 1 (that is, a function of the form $c_{-1}e^{-it} + c_0 + c_{1}e^{it}$, with $c_1 \in \C \setminus \{0\}$), as well the product of all these factors
\[
p(t) := \prod_{v \in V(\PP)} \prod_{j = 1}^{M_v} \prod_{l = 1}^d \innerp{w_{j,l}^v}{z(t)},
\]
as a trigonometric polynomial. Multiplying~\eqref{eq:hatpolytope} by $(2\pi i)^d p(t)$ and using the assumption that $\hat 1_{\PP}(z(t)) = 0$ for all $t \in (-\pi, \pi]$, we get 
\begin{equation}\label{eq:hatpolytope2}
0 = \sum_{v \in V(P)} p_v(t)e^{-2\pi i \innerpsmall{v}{z(t)}},
\end{equation}
where each $p_v(t)$ is also a trigonometric polynomial, since the factors in the denominators of~\eqref{eq:hatpolytope} and in $p(t)$ cancel out. We denote the coefficients of $p_v(t)$ by $c_{v,k}$, as follows:
\begin{equation}\label{eq:pv-trigonometric}
p_v(t) := p(t) \sum_{j = 1}^{M_v} \frac{ \det K_{v, j} }{ \innerp{w_{j,1}^v}{z(t)} \dots \innerp{w_{j,d}^v}{z(t)}} = \sum_{k = -N}^N c_{v,k} e^{ikt}.
\end{equation}
Defining
\[
q_v(t) := \prod_{y \in V(P) \setminus\{v\}} \prod_{j = 1}^{M_y} \, \prod_{l = 1}^d \, \innerp{w_{j,l}^y}{z(t)},
\]
we may write $p_v(t)$ as
\begin{equation*}
p_v(t) = q_v(t) \sum_{j = 1}^{M_v} \det K_{v,j}  \prod_{\substack{k = 1\\ k \neq j}}^{M_v} \prod_{l = 1}^d \innerp{w_{k,l}^v}{z(t)}.
\end{equation*}

To confirm that no cancellation happens and that in particular the functions $p_v(t)$ are not all identically zero, observe that because no edge has both of its first two coordinates equal to zero, the intersection between the subspace of $\R^d$ spanned by the first two coordinates and the spaces orthogonal to each edge is a finite set of lines. Letting $\alpha = re^{i\phi}$ with $r > 0$ and $\phi \in (-\pi, \pi]$,  we may also observe that $e^{-i\phi}z(t) \in \R^d$. Thus we can choose $t_0 \in (-\pi, \pi]$ such that $e^{-i\phi}z(t_0)$ is not orthogonal to any edge. If we define 
\[
u := \mathrm{argmin}_{x \in V(\PP)} \innerp{x}{e^{-i\phi}z(t_0)},
\]
then $\innerp{w_{k,l}^u}{e^{-i\phi}z(t_0)} > 0$ for all $k$ and~$l$. Hence 
\[
\sum_{j = 1}^{M_u} \det K_{u,j}  \prod_{\substack{k = 1\\ k \neq j}}^{M_u} \prod_{l = 1}^d \innerp{w_{k,l}^u}{e^{-i\phi}z(t_0)} = e^{-i\phi d(M_u-1)}\sum_{j = 1}^{M_u} \det K_{u,j}  \prod_{\substack{k = 1\\ k \neq j}}^{M_u} \prod_{l = 1}^d \innerp{w_{k,l}^u}{z(t_0)}  > 0,
\]
and therefore $p_u(t)$ is not identically zero. 

Next, we use the generating functions for the Bessel functions~\eqref{eq:Bessel-generating}. To adapt the formulas for our context, we write the first two coordinates of $v$ in polar form: $v = (r_v\cos\phi_v, r_v\sin\phi_v, v_3, \dots, v_d)$, so that 
\[
-\innerp{v}{z(t)} = -\alpha r_v \cos(t - \phi_v) =  \alpha r_v \sin(t - \phi_v -\pi/2).
\]
Hence from \eqref{eq:Bessel-generating} follows
\begin{align*}
e^{-2\pi i \innerpsmall{v}{z(t)}} =  \sum_{n\in \Z}  J_n(2\pi \alpha r_v) e^{in t} e^{-in( \phi_v +\pi/2 )}.
\end{align*}
Substituting into~\eqref{eq:hatpolytope2},
\[
0 = \sum_{n\in \Z}  \sum_{v \in V(P)} p_v(t) e^{-in( \phi_v +\pi/2 )} J_n(2\pi \alpha  r_v) e^{in t}.
\]
Next we substitute formula~\eqref{eq:pv-trigonometric} into $p_v(t)$ and then replace $n$ by $n-k$ in the summation: 
\begin{align*}
0 &= \sum_{n\in \Z} \sum_{v \in V(P)} \sum_{k = -N}^{N}c_{v,k}  e^{-in( \phi_v +\pi/2 )} J_{n}(2\pi \alpha r_v) e^{i(n+k) t}\\
&= \sum_{n\in \Z} \sum_{v \in V(P)} \sum_{k = -N}^{N}    c_{v,k}  e^{-i(n-k)( \phi_v +\pi/2 )} J_{n-k}(2\pi \alpha r_v) e^{i n t}.  
\end{align*}

The last expression is the Fourier series of the resulting function in $t \in (-\pi, \pi]$, and therefore all of the coefficients must vanish:
\[
\sum_{v \in V(P)} e^{-in \phi_v} \sum_{k = -N}^{N} c_{v,k} \, J_{n-k}(2\pi \alpha r_v)  e^{ik (\phi_v + \pi/2) } = 0.
\]
Using $e^{ik \pi/2} = i^k$, we get the formula from the statement.\qedhere
\end{proof}

\medskip

To prove Theorem~\ref{thm:whenPisapolytope} we will now analyze Equation~\eqref{identity 1, vanishing Bessels-dimd} for large $n$ and determine the asymptotically dominant terms.

\begin{proof}[Proof of Theorem~\ref{thm:whenPisapolytope}]
Let $\PP \subset \R^d$ be a $d$-dimensional polytope, $H$ be a 2-dimensional subspace not orthogonal to any edge from $\PP$ and $e, f \in \R^d$ which form an orthogonal basis for $H$. Suppose, by way of contradiction, that $N(\PP)$ does contain a `circle' $C_\alpha := \{ \alpha (\cos t) e + \alpha (\sin t) f \in \C^d \mid t \in (-\pi, \pi] \}$ for some $\alpha \in \C \setminus \{0\}$.

We may consider a rotation $R$ that sends $H$ to the plane spanned by the first two coordinates of $\R^d$ and observe that $N(\PP)$ contains $C_\alpha$ if and only if $N(R\PP)$ contains $C_\alpha' := \{(\alpha \cos t, \alpha \sin t, 0, \dots, 0) \mid t \in [-\pi, \pi]\}$. The assumption that $H$ is not orthogonal to any edge gets translated to the assumption that no direction $Rw_{j,l}^v$ has both of its first two coordinates equal to zero, and hence we have satisfied the hypotheses of Lemma~\ref{lm:zero relation 1, polytopes}. For simplicity,  we henceforth omit the rotation $R$ and we assume that $\PP$ and $H$ already have this  orientation, in particular $C_\alpha = C_\alpha'$.

By Lemma~\ref{lm:zero relation 1, polytopes}, we know that identity~\eqref{identity 1, vanishing Bessels-dimd} must be true. Since not all of the coefficients $c_{v,k}$ are zero, we may assume that $N$ is the highest degree of a nonzero coefficient and we let $u \in V(\PP)$ be such that $c_{u,N} \neq 0$.  Because a translation of the polytope by a vector $c \in \R^d$ implies that $\hat{1}_{\PP+c}(z) = \hat{1}_{\PP}(z)e^{-2\pi i \innerpsmall{z}{c}}$, we may translate the polytope while preserving the assumption that its null set contains $C_\alpha$. By translating $P$ in the direction of $u$, we may assume that $u = \arg\max_{v \in V} r_v$ and that $u$ is the only vertex that attains this maximum. 

Using the asymptotic~\eqref{eq:bessel-asymptotic} for $J_n(z)$, we have:
\begin{equation}\label{eq:bessel-limit}
\begin{aligned}
\lim_{n \to \infty} \frac{(n-k)!2^{n-k}}{(2\pi r_u\alpha)^{n-k}} J_{n-l}(2\pi r_v\alpha) 
&= \lim_{n \to \infty} \frac{(n-k)!2^{n-k}}{(2\pi r_u\alpha)^{n-k}} \frac{(2\pi r_v\alpha)^{n-l}}{(n-l)!2^{n-l}}\\
&=  
\begin{cases}
1 \text{ if } k = l \text{ and } u = v,\\
0 \text{ if } k < l \text{ or } (k = l \text{ and } u \neq v).
\end{cases}
\end{aligned}
\end{equation}

For any $n > N$, we would like to focus on the unique dominant term of~\eqref{identity 1, vanishing Bessels-dimd}, which grows with $n$ as $\frac{1}{(n-N)!}\left(\frac{2\pi r_u\alpha}{2}\right)^{n-N}$. To be more precise, we  multiply Equation~\eqref{identity 1, vanishing Bessels-dimd} by $e^{in\phi_u}\frac{(n-N)!2^{n-N}}{(2\pi r_u\alpha)^{n-N}}$ to get:
\[
\sum_{v \in V(P)} e^{-in (\phi_v - \phi_u)} \sum_{k = -N}^{N} c_{v,k} \, \frac{(n-N)!2^{n-N}}{(2\pi r_u\alpha)^{n-N}} J_{n-k}(2\pi r_v\alpha)  i^k e^{ik \phi_v } = 0.
\]
Taking the limit as $n \to \infty$, \eqref{eq:bessel-limit} tells us that all terms with $k < N$ and $v \neq u$ tend to $0$, leaving us with only the $k=N$ term:
\[
c_{u, N}\, i^{N}\, e^{iN\phi_u} = 0,
\]
implying that $c_{u,N} = 0$,  a contradiction.

Therefore we conclude that no $\alpha$ can satisfy Equation~\eqref{identity 1, vanishing Bessels-dimd} for every $n$ and hence by Lemma~\ref{lm:zero relation 1, polytopes}, $N(\PP)$ cannot contain $C_\alpha$ for any plane $H$ that is not orthogonal to any edge of $\PP$.
\end{proof}

\bibliographystyle{plain}

\end{document}